\documentclass{amsart}

\usepackage[utf8]{inputenc}

\usepackage{amssymb}
\usepackage{amsmath}
\usepackage{amsthm}
\usepackage{enumitem}
\usepackage{pgf,tikz}
\usetikzlibrary{arrows}

\numberwithin{equation}{section}

\theoremstyle{plain}
\newtheorem{thm}{Theorem}[section]
\newtheorem{lemma}[thm]{Lemma}
\newtheorem{prop}[thm]{Proposition}
\newtheorem{coroll}[thm]{Corollary}

\theoremstyle{definition}
\newtheorem{defn}[thm]{Definition}

\newtheorem{problem}[thm]{Problem}
\newtheorem{question}[thm]{Question}
\newtheorem{remark}[thm]{Remark}
\newtheorem{ex}[thm]{Example}

\usepackage{xspace} % A makrok utani space korrekt kezelese miatt.

\newcommand{\R}{\mathbb{R}}
\newcommand{\Z}{\mathbb{Z}}

\DeclareMathOperator{\conv}{Conv}

\def\Conv{{\rm Conv}}
\def\cG{{\mathcal{G}}}

\DeclareMathOperator{\Vis}{Vis}

\begin{document}

\title[$h^*$-vectors of graph polytopes using activities]{$h^*$-vectors of graph polytopes using activities of dissecting spanning trees}

\author{Tam\'as K\'alm\'an}%\fnref{TK}}
\address{Department of Mathematics, 
Tokyo Institute of Technology\\
H-214, 2-12-1 Ookayama, Meguro-ku, Tokyo 152-8551, Japan}
\email{kalman@math.titech.ac.jp}

\author{Lilla T\'othm\'er\'esz}%\fnref{LT}}
\address{ELKH-ELTE Egerv\'ary Research Group and ELTE Eötvös Loránd University, P\'azm\'any P\'eter s\'et\'any 1/C, Budapest, Hungary}
\email{lilla.tothmeresz@ttk.elte.hu}

\begin{abstract}
    Symmetric edge polytopes of graphs and root polytopes of semi-balanced digraphs are two classes of lattice polytopes whose
    $h^*$-polynomials have interesting properties and generalize important graph polynomials.
	For both classes of polytopes there are large, natural classes of dissections into unimodular simplices. These are such that the simplices correspond to certain spanning trees.
	
	We show that for any ``spanning tree dissection'' of the symmetric edge polytope of a graph, or the root polytope of a semi-balanced digraph, the $h^*$-polynomial
	of the polytope 
	can be computed as a generating function of certain activities of the corresponding spanning trees.
    Apart from giving simple and flexible algorithms
	for computing these polynomials, our results also reveal that all dissections in question are surprisingly similar to each other: It turns out that 
	the distributions of
	many statistics of spanning tree dissections are in fact independent of the actual dissection.
\end{abstract}

%\begin{keyword} symmetric edge polytope \sep root polytope \sep $h^*$-vector \sep activity 
%	\MSC[2020] 52B20 \sep 05C31 \sep 52B45
%\end{keyword}

\maketitle

\section{Introduction}
\label{s:intro}

\subsection{Aim of the paper}
Let $\cG$ be a simple undirected graph with vertex set $V(\cG)$ and edge set $E(\cG)$.
The polyhedron $$P_\cG =\conv \{\mathbf{1}_u-\mathbf{1}_v, \mathbf{1}_v-\mathbf{1}_u \mid uv\in E(\cG)\} \subset \mathbb{R}^{V(\mathcal{G})}$$ is called the \emph{symmetric edge polytope} of $\cG$ \cite{Sym_edge_appearance}.
(Here $\mathbf{1}_u, \mathbf{1}_v$ stand for generators of $\mathbb R^{V(\cG)}$ that correspond to $u,v\in V(\cG)$.)
Symmetric edge polytopes have recently garnered considerable interest \cite{arithm_symedgepoly,DDM19,Sym_edge_appearance,OT}. Most of the literature concentrates on the computation and properties of their $h^*$-vectors (a.k.a.\ $h^*$-polynomials). 
In particular, the $h^*$-vector of a symmetric edge polytope is conjectured to be $\gamma$-positive \cite{OT}.

Directed graphs (digraphs) have
an associated graph polytope with a definition similar to that of the symmetric edge polytope \cite{alex}. More concretely, the \emph{root polytope} of a directed graph $G=(V,E)$ is the convex hull
\[\mathcal{Q}_G=\conv\{\,\mathbf{1}_h - \mathbf{1}_t \mid \overrightarrow{th}\in E\,\}.\] 
In fact, symmetric edge polytopes are root polytopes of bidirected graphs, where a bidirected graph is the directed graph obtained from an undirected graph by replacing each edge $uv$ with two directed edges $\overrightarrow{uv}$ and $\overrightarrow{vu}$.

We will be interested in root polytopes of so-called semi-balanced digraphs, which are directed graphs so that each cycle has the same number of edges oriented in the two cyclic directions. 
There are two reasons to consider this class of digraphs: It turns out that connected semi-balanced digraphs are exactly the connected digraphs where the dimension of $\mathcal{Q}_G$ is $|V(G)|-2$ (otherwise the dimension is $|V(G)|-1$) \cite{semibalanced}, that is, $\mathcal{Q}_G$ has rather special properties for semi-balanced digraphs.
Moreover, %it turns out that 
all facets of symmetric edge polytopes are root polytopes of certain semi-balanced digraphs \cite{arithm_symedgepoly}, i.e., these polytopes are relevant also for the investigation of symmetric edge polytopes. The $h^*$-polynomial of $\mathcal{Q}_G$ is called the \emph{interior polynomial} of $G$ in \cite{semibalanced}. This polynomial generalizes the interior polynomial of a hypergraph, as well as the greedoid polynomial of a planar branching greedoid. Here the interior polynomial of a hypergraph is itself a generalization of $T(x,1)$ (associated to a graph), where $T(x,y)$ is the Tutte polynomial.

In this paper, 
we give graph-theoretic formulas for the 
$h^*$-polynomials of $P_\mathcal G$ and $\mathcal Q_G$ (for arbitrary $\cG$ and semi-balanced $G$), namely we express them as generating functions of certain activity statistics on certain spanning trees. 
We also pose some open questions.

The paper is structured as follows.  In Section \ref{ss:setup_results}, we informally explain the setup and our results. 
In Section \ref{sec:prelim} we introduce the necessary tools. We prove our theorems in Section \ref{sec:proofs}. Finally, in Section \ref{s:open_problems}, we mention some open problems and further explain the motivation for our work. 

Acknowledgements:
TK was supported by a Japan Society for the Promotion of Science (JSPS) Grant-in-Aid for Scientific Research C (no.\ 17K05244).
LT was supported by the National Research, Development and Innovation Office of Hungary -- NKFIH, grant no.\ 132488, by the János Bolyai Research Scholarship of the Hungarian Academy of Sciences, and by the ÚNKP-21-5 New National Excellence Program of the Ministry for Innovation and Technology, Hungary. LT was also partially supported by the Counting in Sparse Graphs Lendület Research Group of the Alfr\'ed Rényi Institute of Mathematics.

\subsection{Setup and results}\label{ss:setup_results}
For both the symmetric edge polytope of a graph and the root polytope of a semi-balanced digraph, $h^*$-polynomials are typically computed by first dissecting the polytope, using only the vertices of the polytope (and, in the former case, the origin) as $0$-simplices. Here a \emph{dissection into simplices} means a set of maximal dimensional
simplices with disjoint interiors, so that their union is the whole polytope. A \emph{triangulation} is a dissection into simplices where we additionally require that the intersection of any two simplices be a common face. Triangulations are more often used than dissections, but everything we say in this paper applies also to the broader case of dissections.

The vertices of $\mathcal{Q}_G$ correspond to the (directed) edges of $G$.
For a %connected 
semi-balanced digraph $G$, 
a set of vertices is affine independent if and only if the corresponding subgraph is a forest \cite{semibalanced}. (Thus if $G$ is connected then  $\dim\mathcal{Q}_G=|V(G)|-2$. Without semi-balancedness, the `only if' direction fails.)
Therefore, if we want to 
dissect 
$\mathcal{Q}_G$ using only its vertices, maximal simplices will correspond to spanning trees.
In other words, dissecting the root polytope $\mathcal{Q}_G$ amounts to finding a subset $\mathcal{T}$ of the spanning trees of $G$ that satisfies the requirement that the corresponding simplices $\{\mathcal{Q}_T \mid T\in \mathcal{T}\}$ are interior disjoint and cover the polytope. We will call such a set $\mathcal{T}$ of spanning trees a \emph{dissecting tree set for $G$}. (See Figure \ref{fig:semi-balanced_ex} for an example.)

It is easy to show that for each spanning tree $T$, the simplex $\mathcal{Q}_T$ is unimodular
\cite{semibalanced}.
This turns out to be crucial for the computation that we intend to carry out. In fact, the abundance of unimodular simplices within $\mathcal{Q}_G$ and $P_\cG$, and their scarcity in general, explains why we only consider these particular polytopes. We have of course well-established methods for finding (say, regular) 
triangulations of polytopes.
There are also natural ways of producing 
dissections (which in general fail to be triangulations) for $\mathcal{Q}_G$ and $P_\cG$.
One such construction is based on ribbon structures \cite{semibalanced,sym_ribbon}. Arguably, its most important special case is the following.

\begin{ex}
\label{ex:coeulerian} 
Let $G$ be a connected planar semi-balanced digraph, with planar dual $G^*$, and let $r$ be an arbitrary vertex of $G^*$. Notice that in this case $G^*$ %is 
naturally becomes
an Eulerian digraph. 
It is well known in general that the spanning trees of $G$ are exactly the complements of the spanning trees of $G^*$. Now the complements of the spanning arborescences of $G^*$, rooted at $r$, form a dissecting (in fact, triangulating) tree set for $G$ \cite{semibalanced}.
\end{ex}

To the symmetric edge polytope $P_\mathcal{G}$ of an undirected graph $\mathcal{G}$, each edge of $\cG$ contributes two vertices: they correspond to the two orientations of the edge.
If $\cG$ is connected then
$P_\cG$ has dimension $|V(\cG)|-1$. As alluded to earlier, the facets of $P_{\mathcal{G}}$ are root polytopes of spanning 
subgraphs of $\mathcal{G}$ that are oriented in a semi-balanced way. As the origin is in the relative interior of $P_\mathcal{G}$, one can dissect
$P_\mathcal{G}$ by first dissecting the 
facets of the boundary, then coning
at the origin. For any directed graph $G$, let $\tilde{\mathcal{Q}}_G=\conv(\{\mathbf{0}\} \cup \{\,\mathbf{1}_h - \mathbf{1}_t\mid \overrightarrow{th}\in E(G)\,\})$. 
Then our strategy for dissecting $P_\mathcal{G}$ calls for appropriately
taking simplices of the type $\tilde{\mathcal{Q}}_T$ for some oriented spanning trees of $\cG$.
Here again, any such dissection is automatically unimodular.
For an undirected graph $\mathcal{G}$, we will call a set of oriented spanning trees $\mathcal{T}$ a \emph{dissecting tree set for $\cG$} if the corresponding simplices $\mathcal{Q}_T$ dissect the boundary of $P_{\cG}$.
Dissections, and even triangulations, of this type are not hard to construct and this has been done in various concrete ways \cite{arithm_symedgepoly,sym_ribbon}.

We need one more definition in order to state our first result.
Let $T$ be a directed tree, and let $v$ be a fixed vertex.
If $e\in T$ is an edge, then $T-e$ has two components. We say that $e\in T$ is \emph{pointing away from $v$ in $T$} if $v$ is in the component of $T-e$ that contains the tail of $e$.

\begin{thm}\label{thm:sym_edge_h-vector}
 Let $\mathcal G$ be a connected graph and
 let us denote the $h^*$-vector of the symmetric edge polytope
 $P_\mathcal{G}$ by $h^*_\cG$. Let $\mathcal{T}$ be any dissecting tree set for 
	%the connected undirected graph
	$\cG$, and let $v$ be any vertex of $\cG$.
	Then
	$$(h^*_\cG)_i = |\{T\in \mathcal{T}\mid T \text{ has exactly $i$ edges pointing away from $v$}\}|.$$
\end{thm}

Let us stress again that even though $\cG$ is unoriented, its dissecting trees $T$ are oriented.
We prove Theorem \ref{thm:sym_edge_h-vector} in section  \ref{sec:proofs}. We note that the (short) proof follows closely the proof of \cite[Proposition 4.6]{arithm_symedgepoly} where this formula is established for the $h$-vector of a certain triangulation of the symmetric edge polytope of a complete bipartite graph.

As a corollary we obtain an interesting identity, namely that the distribution of the number of edges pointing away from an arbitrary given node is independent of which node and which dissecting tree set we consider.

\begin{coroll}
	Let $\cG$ be an undirected graph, and let 
	%$i\in \{0, \dots, |V(\mathcal{G})|-1\}$ 
	$0\le i\le|V(\mathcal{G})|-1$
	be arbitrary.
	For any dissecting tree set $\mathcal{T}$ for $\cG$ and any fixed vertex $v\in V(\mathcal{G})$, the value $$|\{T\in \mathcal{T}\mid T \text{ has exactly $i$ edges pointing away from $v$}\}|$$ is the same.
\end{coroll}

\begin{ex}\label{ex:K_3_computation}
The upper left panel of Figure \ref{fig:triangle_ex} shows the undirected triangle $K_3$, drawn as a bidirected graph. $P_{K_3}$ has dimension $2$, and in fact it is easy to check that it is a regular hexagon. The six facets of $P_{K_3}$ correspond to the six subgraphs drawn by thick edges in panels $2$ through $7$. All of these facets are simplices, hence the six spanning trees in the last six panels form a dissecting tree set for $K_3$ (which happens to be unique in this case).

One can easily check that whichever vertex we choose as base vertex, there is one tree with $0$ edges pointing away from the base vertex, four trees with $1$ edge pointing away from the base vertex, and one tree with $2$ edges pointing away from the base vertex. Hence by Theorem \ref{thm:sym_edge_h-vector}, the $h^*$-polynomial of the symmetric edge polytope is $x^2+4x+1$.
\end{ex}

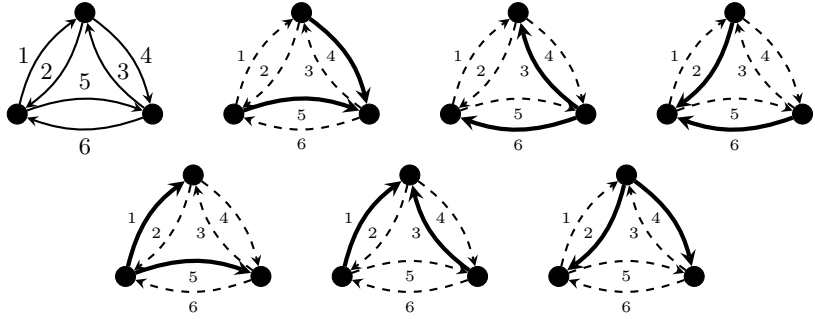
\begin{figure}
    \begin{center}
    \begin{tikzpicture}[scale=0.9]
    \tikzstyle{o}=[circle,fill,scale=.8,draw]
	\begin{scope}[shift={(-4.8,0)}]
	\node [o] (1) at (0,0) {};
	\node [o] (2) at (1,1.5) {};
	\node [o] (3) at (2,0) {};
	
	%\node [] (7) at (2.3,-0.1) {$v_0$};	
	\path [thick,->,>=stealth,] (1) edge [left,bend left=20] node {\small 1} (2);
	\path [thick,->,>=stealth,] (2) edge [left,bend left=20] node {\small 2} (1);
	\path [thick,->,>=stealth,] (1) edge [above,bend left=20] node {\small 5} (3);
	\path [thick,->,>=stealth,] (3) edge [below,bend left=20] node {\small 6} (1);
	\path [thick,->,>=stealth,] (2) edge [right,bend left=20] node {\small 4} (3);
	\path [thick,->,>=stealth,] (3) edge [right,bend left=20] node {\small 3} (2);
	\end{scope}
	
	\begin{scope}[shift={(-1.6,0)}]
	\node [o] (1) at (0,0) {};
	\node [o] (2) at (1,1.5) {};
	\node [o] (3) at (2,0) {};
	
	%\node [] (7) at (2.3,-0.1) {$v_0$};	
	\path [thick,->,>=stealth,dashed] (1) edge [left,bend left=20] node {\tiny 1} (2);
	\path [thick,->,>=stealth,dashed] (2) edge [left,bend left=20] node {\tiny 2} (1);
	\path [ultra thick,->,>=stealth,] (1) edge [below,bend left=20] node {\tiny 5} (3);
	\path [thick,->,>=stealth,dashed] (3) edge [below,bend left=20] node {\tiny 6} (1);
	\path [ultra thick,->,>=stealth,] (2) edge [left,bend left=20] node {\tiny 4} (3);
	\path [thick,->,>=stealth,dashed] (3) edge [left,bend left=20] node {\tiny 3} (2);
	\end{scope}
	\begin{scope}[shift={(1.6,0)}]
	\node [o] (1) at (0,0) {};
	\node [o] (2) at (1,1.5) {};
	\node [o] (3) at (2,0) {};
	
	%\node [] (7) at (2.3,-0.1) {$v_0$};	
	\path [thick,->,>=stealth,dashed] (1) edge [left,bend left=20] node {\tiny 1} (2);
	\path [thick,->,>=stealth,dashed] (2) edge [left,bend left=20] node {\tiny 2} (1);
	\path [thick,->,>=stealth,dashed] (1) edge [below,bend left=20] node {\tiny 5} (3);
	\path [ultra thick,->,>=stealth,] (3) edge [below,bend left=20] node {\tiny 6} (1);
	\path [thick,->,>=stealth,dashed] (2) edge [left,bend left=20] node {\tiny 4} (3);
	\path [ultra thick,->,>=stealth,] (3) edge [left,bend left=20] node {\tiny 3} (2);
	\end{scope}
	\begin{scope}[shift={(4.8,0)}]
	\node [o] (1) at (0,0) {};
	\node [o] (2) at (1,1.5) {};
	\node [o] (3) at (2,0) {};
	
	%\node [] (7) at (2.3,-0.1) {$v_0$};	
	\path [thick,->,>=stealth,dashed] (1) edge [left,bend left=20] node {\tiny 1} (2);
	\path [ultra thick,->,>=stealth,] (2) edge [left,bend left=20] node {\tiny 2} (1);
	\path [thick,->,>=stealth,dashed] (1) edge [below,bend left=20] node {\tiny 5} (3);
	\path [ultra thick,->,>=stealth,] (3) edge [below,bend left=20] node {\tiny 6} (1);
	\path [thick,->,>=stealth,dashed] (2) edge [left,bend left=20] node {\tiny 4} (3);
	\path [thick,->,>=stealth,dashed] (3) edge [left,bend left=20] node {\tiny 3} (2);
	\end{scope}\begin{scope}[shift={(-3.2,-2.4)}]
	\node [o] (1) at (0,0) {};
	\node [o] (2) at (1,1.5) {};
	\node [o] (3) at (2,0) {};
	
	%\node [] (7) at (2.3,-0.1) {$v_0$};	
	\path [ultra thick,->,>=stealth,] (1) edge [left,bend left=20] node {\tiny 1} (2);
	\path [thick,->,>=stealth,dashed] (2) edge [left,bend left=20] node {\tiny 2} (1);
	\path [ultra thick,->,>=stealth,] (1) edge [below,bend left=20] node {\tiny 5} (3);
	\path [thick,->,>=stealth,dashed] (3) edge [below,bend left=20] node {\tiny 6} (1);
	\path [thick,->,>=stealth,dashed] (2) edge [left,bend left=20] node {\tiny 4} (3);
	\path [thick,->,>=stealth,dashed] (3) edge [left,bend left=20] node {\tiny 3} (2);
	\end{scope}
	\begin{scope}[shift={(0,-2.4)}]
	\node [o] (1) at (0,0) {};
	\node [o] (2) at (1,1.5) {};
	\node [o] (3) at (2,0) {};
	
	%\node [] (7) at (2.3,-0.1) {$v_0$};	
	\path [ultra thick,->,>=stealth,] (1) edge [left,bend left=20] node {\tiny 1} (2);
	\path [thick,->,>=stealth,dashed] (2) edge [left,bend left=20] node {\tiny 2} (1);
	\path [thick,->,>=stealth,dashed] (1) edge [below,bend left=20] node {\tiny 5} (3);
	\path [thick,->,>=stealth,dashed] (3) edge [below,bend left=20] node {\tiny 6} (1);
	\path [thick,->,>=stealth,dashed] (2) edge [left,bend left=20] node {\tiny 4} (3);
	\path [ultra thick,->,>=stealth,] (3) edge [left,bend left=20] node {\tiny 3} (2);
	\end{scope}
	\begin{scope}[shift={(3.2,-2.4)}]
	\node [o] (1) at (0,0) {};
	\node [o] (2) at (1,1.5) {};
	\node [o] (3) at (2,0) {};
	
	%\node [] (7) at (2.3,-0.1) {$v_0$};	
	\path [thick,->,>=stealth,dashed] (1) edge [left,bend left=20] node {\tiny 1} (2);
	\path [ultra thick,->,>=stealth,] (2) edge [left,bend left=20] node {\tiny 2} (1);
	\path [thick,->,>=stealth,dashed] (1) edge [below,bend left=20] node {\tiny 5} (3);
	\path [thick,->,>=stealth,dashed] (3) edge [below,bend left=20] node {\tiny 6} (1);
	\path [ultra thick,->,>=stealth,] (2) edge [left,bend left=20] node {\tiny 4} (3);
	\path [thick,->,>=stealth,dashed] (3) edge [left,bend left=20] node {\tiny 3} (2);
	\end{scope}
	\end{tikzpicture}
	\end{center}
    \caption{The complete graph $K_3$, drawn as a bidirected graph, and a dissecting tree set.}
    \label{fig:triangle_ex}
\end{figure}

For our other results we need some further preparation on directed graphs.
Let $G$ be an arbitrary digraph, and let $T$ be a 
spanning tree of $G$.
For an edge $e\in T$, the \emph{fundamental cut} of $e$ with respect to $T$, denoted by $C^*(T,e)$, is the set of edges of $G$ connecting the two components of $T-e$ (which are called the \emph{shores} of the cut).
We say that an edge $e'\in C^*(T,e)$
\emph{stands parallel} to $e$ if the heads of $e$ and $e'$ are in the same shore. 
Otherwise we say that $e'$ \emph{stands opposite} to $e$. 

\begin{ex}
In the second panel of Figure \ref{fig:triangle_ex}, the edges labeled $4$ and $5$ form a spanning tree (thick edges). The fundamental cut of the edge 4 is $\{1,2,3,4\}$, with 2 and 4 standing parallel to 4, and 1 and 3 standing opposite to it.
\end{ex}

We will also need the following notion.

\begin{defn}[internal semi-activity in digraphs \cite{hyperBernardi}]
\label{def:semiactive}
	Let $G$ be a digraph with a fixed ordering of the edges. Let $T$ be a spanning tree of $G$. An 
	edge $e\in T$ is \emph{internally semi-active} for $T$ if in the fundamental cut $C^*(T,e)$, the maximal edge (with respect to the fixed ordering) stands parallel to $e$. If the maximal edge stands opposite to $e$, then we say that $e$ is \emph{internally semi-passive} for $T$. 
	
	The \emph{internal semi-activity} of a spanning tree (with respect to the fixed order) is the number of its internally semi-active edges, while the \emph{internal semi-passivity} is the number of internally semi-passive edges.
\end{defn}

This notion of activity is the dual pair of 
``external semi-activity'' 
that was privately communicated to us by Alex Postnikov. Internal semi-activity is similar to Tutte's concept of internal activity \cite{tutte}, but instead of requiring $e\in T$ to be the maximal element in $C^*(T,e)$, it only requires $e$ to stand parallel to the maximal edge of $C^*(T,e)$.

\begin{ex}
The second panel of Figure \ref{fig:triangle_ex} shows the spanning tree $\{4,5\}$. We use the ordering of the edges given by the labeling. The edge $4$ is the maximal element in its fundamental cut, and each edge is parallel to itself, thus $4$ is internally semi-active. On the other hand, the maximal element of the fundamental cut of $5$ is $6$, which stands opposite to $5$, whence $5$ is internally semi-passive.
\end{ex}

These concepts lead to our second expression for the $h^*$-polynomial of $P_\cG$. In order for them to apply, 
we need to think of $\cG$ as a bidirected graph. That is, for any $uv\in E(\cG)$, we take two directed edges $\overrightarrow{uv}$ and $\overrightarrow{vu}$. 

\begin{thm}\label{thm:sym_edge_h-vector_fixed_order}
	%Let us denote the $h^*$-vector of $P_\mathcal{G}$ by $h^*_\cG$. 
	Let $\mathcal{T}$ be any dissecting tree set for 
	the undirected graph
	$\cG$. Consider $\cG$ as a bidirected graph, and fix an ordering of the (directed) edges of $\cG$.
	Then 
	the $h^*$-vector of the symmetric edge polytope
 $P_\mathcal{G}$ satisfies
	$$(h^*_\cG)_i = |\{T\in \mathcal{T}\mid T \text{ has exactly $i$ internally semi-passive edges}\}|.$$
\end{thm}

\begin{ex}
Regarding Figure \ref{fig:triangle_ex}, using the labeling of the edges as the ordering, the internal semi-passivities of the spanning trees of panels 2 through 7 are, respectively, $1$, $1$, $0$, $2$, $1$, and $1$. Hence we again (cf.\ Example \ref{ex:K_3_computation}) conclude that the $h^*$-polynomial of the symmetric edge polytope of $K_3$ is $x^2+4x+1$. 
\end{ex}

We now turn our attention to root polytopes of semi-balanced digraphs, where we obtain literally the same formula. 

\begin{thm}\label{thm:interior_poly_fixed_edge_order}
	Let $G$ be a semi-balanced digraph and let us denote the $h^*$-vector of its root polytope $\mathcal{Q}_G$ by $h^*_{\mathcal{Q}_G}$. Let $\mathcal{T}$ be any set of dissecting spanning trees for $G$, and fix any ordering of the edges of $G$.
	Then
	$$(h^*_{\mathcal{Q}_G})_i = |\{T\in \mathcal{T}\mid T \text{ has exactly $i$ internally semi-passive edges}\}|.$$
\end{thm}

\begin{ex}\label{ex:semibalanced_fixed}
Figure \ref{fig:semi-balanced_ex} shows a semi-balanced digraph in its left panel, and four spanning trees in panels 2 through 5 that form a dissecting tree set. To check that 
this is so,
%the four trees indeed form a dissecting tree set, 
note that these are exactly the Jaeger trees (see subsection \ref{ss:embedding_passivities} for the definition) 
of the digraph if we choose the ribbon structure induced by the positive orientation of the plane, moreover, the base vertex is the lower left vertex and the base edge is the one connecting the lower left and lower right vertices \cite[Theorem 5.8]{semibalanced}. 
In fact, our set of four trees is even triangulating \cite[Theorem 1.1]{KM}.

For the edge ordering indicated in the figure, the trees $T_1,T_2,T_3,T_4$ have internal semi-passivities $1,0,2,1$, respectively. Hence the $h^*$-polynomial of the root polytope is $x^2+2x+1$.
\end{ex}

\begin{figure}
	\begin{center}
	\begin{tikzpicture}[scale=.2]
	\begin{scope}[shift={(-20,0)}]
	%\node [] (0) at (-2,10) {$s$};
	\node [circle,fill,scale=.5,draw] (1) at (0,10) {};
	\node [circle,fill,scale=.5,draw] (2) at (5,10) {};
	\node [circle,fill,scale=.5,draw] (3) at (0,5) {};
	\node [circle,fill,scale=.5,draw] (4) at (5,5) {};
	\node [circle,fill,scale=.5,draw] (5) at (0,0) {};		
	\node [circle,fill,scale=.5,draw] (6) at (5,0) {};		
	%\node [] (7) at (20,7) {$v$};	
	\path [thick,<-,>=stealth] (1) edge [above] node {\small 1} (2);
	\path [thick,->,>=stealth] (3) edge [above] node {\small 2} (4);
	\path [thick,->,>=stealth] (5) edge [above] node {\small 3} (6);
	\path [thick,->,>=stealth] (3) edge [left] node {\small 4} (1);
	\path [thick,<-,>=stealth] (4) edge [right] node {\small 5} (2);
	\path [thick,->,>=stealth] (5) edge [left] node {\small 6} (3);
	\path [thick,->,>=stealth] (6) edge [right] node {\small 7} (4);
	\node [] (0) at (2.5,-4) {\small $G$};
	\end{scope}
	\begin{scope}[shift={(-10,0)}]
	%\node [] (0) at (-2,10) {$s$};
	\node [circle,fill,scale=.5,draw] (1) at (0,10) {};
	\node [circle,fill,scale=.5,draw] (2) at (5,10) {};
	\node [circle,fill,scale=.5,draw] (3) at (0,5) {};
	\node [circle,fill,scale=.5,draw] (4) at (5,5) {};
	\node [circle,fill,scale=.5,draw] (5) at (0,0) {};		
	\node [circle,fill,scale=.5,draw] (6) at (5,0) {};		
	%\node [] (7) at (20,7) {$v$};	
	\path [ultra thick,<-,>=stealth] (1) edge [left] node {} (2);
	\path [thick,dashed,->,>=stealth] (3) edge [above] node {} (4);
	\path [thick,dashed,->,>=stealth] (5) edge [right] node {} (6);
	\path [ultra thick,->,>=stealth] (3) edge [below] node {} (1);
	\path [ultra thick,<-,>=stealth] (4) edge [right] node {} (2);
	\path [ultra thick,->,>=stealth] (5) edge [below] node {} (3);
	\path [ultra thick,->,>=stealth] (6) edge [above] node {} (4);
	\node [] (0) at (2.5,-4) {\small $T_1$};
	\end{scope}
	\begin{scope}[shift={(0,0)}]
	%\node [] (0) at (-2,10) {$s$};
	\node [circle,fill,scale=.5,draw] (1) at (0,10) {};
	\node [circle,fill,scale=.5,draw] (2) at (5,10) {};
	\node [circle,fill,scale=.5,draw] (3) at (0,5) {};
	\node [circle,fill,scale=.5,draw] (4) at (5,5) {};
	\node [circle,fill,scale=.5,draw] (5) at (0,0) {};		
	\node [circle,fill,scale=.5,draw] (6) at (5,0) {};		
	%\node [] (7) at (20,7) {$v$};	
	\path [thick,dashed,<-,>=stealth] (1) edge [left] node {} (2);
	\path [ultra thick,->,>=stealth] (3) edge [above] node {} (4);
	\path [thick,->,>=stealth,dashed] (5) edge [right] node {} (6);
	\path [ultra thick,->,>=stealth] (3) edge [below] node {} (1);
	\path [ultra thick,<-,>=stealth] (4) edge [right] node {} (2);
	\path [ultra thick,->,>=stealth] (5) edge [below] node {} (3);
	\path [ultra thick,->,>=stealth] (6) edge [above] node {} (4);
	\node [] (0) at (2.5,-4) {\small $T_2$};
	\end{scope}
	\begin{scope}[shift={(10,0)}]
	%\node [] (0) at (-2,10) {$s$};
	\node [circle,fill,scale=.5,draw] (1) at (0,10) {};
	\node [circle,fill,scale=.5,draw] (2) at (5,10) {};
	\node [circle,fill,scale=.5,draw] (3) at (0,5) {};
	\node [circle,fill,scale=.5,draw] (4) at (5,5) {};
	\node [circle,fill,scale=.5,draw] (5) at (0,0) {};		
	\node [circle,fill,scale=.5,draw] (6) at (5,0) {};		
	%\node [] (7) at (20,7) {$v$};	
	\path [ultra thick,<-,>=stealth] (1) edge [left] node {} (2);
	\path [thick,dashed,->,>=stealth] (3) edge [above] node {} (4);
	\path [ultra thick,->,>=stealth] (5) edge [right] node {} (6);
	\path [ultra thick,->,>=stealth] (3) edge [below] node {} (1);
	\path [ultra thick,<-,>=stealth] (4) edge [right] node {} (2);
	\path [ultra thick,->,>=stealth] (5) edge [below] node {} (3);
	\path [thick,dashed,->,>=stealth] (6) edge [above] node {} (4);
	\node [] (0) at (2.5,-4) {\small $T_3$};
	\end{scope}
	\begin{scope}[shift={(20,0)}]
	%\node [] (0) at (-2,10) {$s$};
	\node [circle,fill,scale=.5,draw] (1) at (0,10) {};
	\node [circle,fill,scale=.5,draw] (2) at (5,10) {};
	\node [circle,fill,scale=.5,draw] (3) at (0,5) {};
	\node [circle,fill,scale=.5,draw] (4) at (5,5) {};
	\node [circle,fill,scale=.5,draw] (5) at (0,0) {};		
	\node [circle,fill,scale=.5,draw] (6) at (5,0) {};		
	%\node [] (7) at (20,7) {$v$};	
	\path [thick,dashed,<-,>=stealth] (1) edge [left] node {} (2);
	\path [ultra thick,->,>=stealth] (3) edge [above] node {} (4);
	\path [ultra thick,->,>=stealth] (5) edge [right] node {} (6);
	\path [ultra thick,->,>=stealth] (3) edge [below] node {} (1);
	\path [ultra thick,<-,>=stealth] (4) edge [right] node {} (2);
	\path [ultra thick,->,>=stealth] (5) edge [below] node {} (3);
	\path [thick,->,>=stealth,dashed] (6) edge [above] node {} (4);
	\node [] (0) at (2.5,-4) {\small $T_4$};
	\end{scope}
	\end{tikzpicture}
	\end{center}
	\caption{A semi-balanced digraph (first panel), and four of its spanning trees (panels 2-5) that form a dissecting tree set.}
	\label{fig:semi-balanced_ex}
\end{figure}
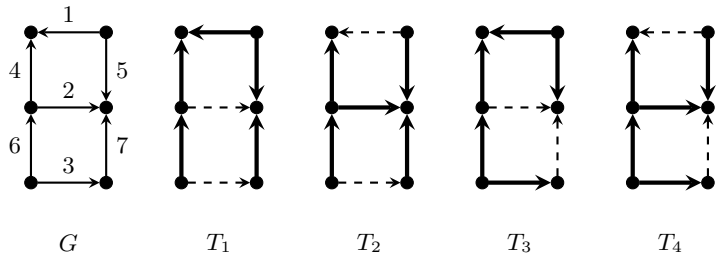

The proofs of Theorems 
\ref{thm:sym_edge_h-vector_fixed_order} and 
\ref{thm:interior_poly_fixed_edge_order}
are almost identical, too. In fact, even the proof of Theorem \ref{thm:sym_edge_h-vector} uses the same underlying principle, although in a substantially different way. That principle, cf.\ Lemma \ref{lem:visibility_disjoint} and Proposition  \ref{prop:h^*-vector_from_visible_facets}, was discerned from Higashitani et al.'s work \cite{arithm_symedgepoly}, where it appears under more restrictive conditions. It states that if we fix a dissection and a generic interior point $\mathbf q$ of the polytope $P$, then the parts of the simplices that are invisible from $\mathbf q$ form a decomposition of $P$; furthermore, if the simplices are unimodular, then the decomposition can be used to compute the $h^*$-vector of $P$. With that, our proofs become a matter of finding the right $\mathbf q$ so that `visibility' takes on a suitable interpretation.

As corollaries to Theorems \ref{thm:sym_edge_h-vector_fixed_order} and \ref{thm:interior_poly_fixed_edge_order}, we obtain that both for symmetric edge polytopes and for root polytopes of semi-balanced digraphs, the distribution of internal semi-passivity over a dissecting tree set is independent of the edge ordering, and even the actual dissection.

\begin{coroll}
    Let $G$ be a bidirected graph and let $i\in \{0, \dots, |V(G)|-1\}$ be arbitrary.
	For any dissecting tree set $\mathcal{T}$, and any ordering of the (directed) edges of $G$, the value $$|\{T\in \mathcal{T}\mid T \text{ has exactly $i$ internally semi-passive edges}\}|$$ is the same.
\end{coroll}

\begin{coroll}
    Let $G$ be a semi-balanced digraph. 
	Let $i\in \{0, \dots, |V(G)|-1\}$ be arbitrary.
	For any dissecting tree set $\mathcal{T}$, and any ordering of the edges of $G$, the value $$|\{T\in \mathcal{T}\mid T \text{ has exactly $i$ internally semi-passive edges}\}|$$ is the same.
\end{coroll}

We end this introduction by drawing attention to a certain dichotomy of approaches to activity statistics, and by positioning our work in it. 
The Tutte polynomial was originally constructed as the generating function of activities with respect to a fixed edge order \cite{tutte}. On the other hand, Bernardi presents the Tutte polynomial as a generating function of embedding activities \cite{Bernardi_Tutte} that depend not on a fixed order but on a ribbon structure. 
%That is, in his case the `auxiliary data' are not of a fixed 
Of course, neither type of auxiliary data influences the final outcome. Also, neither definition is a special case of the other.

In the previous papers \cite{semibalanced,sym_ribbon} we gave formulas for the $h^*$-polynomials of symmetric edge polytopes and of root polytopes (of semi-balanced digraphs) that were inspired by Bernardi's ideas.
The results of this paper, especially Theorems \ref{thm:sym_edge_h-vector_fixed_order} and 
\ref{thm:interior_poly_fixed_edge_order}, follow Tutte's approach instead. Between the formulas of \cite{semibalanced,sym_ribbon} and those of this paper, neither generalizes the other. However, in some sense, the claims of this paper are more robust because they work for \emph{any} dissecting tree set, while in \cite{semibalanced,sym_ribbon} we considered a specific dissecting tree set (the aforementioned Jaeger trees) that also depended on the embedding. In Section \ref{s:open_problems}, we explore this ``lack of robustness'' in more detail, and pose it as an open question to explain its background.

\section{Preliminaries on Ehrhart theory and $h^*$-polynomials}\label{sec:prelim}

Suppose that $P\subset\R^n$ is a $d$-dimensional lattice polytope, that is a 
polytope with vertices in $\Z^n$. 
Its \emph{Ehrhart polynomial} $\varepsilon_P$ associates, to each $t\in \mathbb{N}$, the point count $\varepsilon_P(t)=|(t\cdot P) \cap \mathbb{Z}^n|$. This is indeed known to be a polynomial in $t$ (which then can be extended to arbitrary reals). 

It is easy to check that the polynomials $C_k(t)=\binom{t+d-k}{d}$, for $k=0, \dots, d$, provide a basis over $\mathbb{Q}$ for the space of 
%at most degree $d$ 
polynomials, of degree $d$ or lower, with rational coefficients \cite[Lemma 3.8]{KP_Ehrhart}. Hence one can uniquely write the Ehrhart polynomial as
$$\varepsilon_P(t) = \sum_{k=0}^d a_k C_k(t),$$
and define the \emph{$h^*$-polynomial} (or \emph{$h^*$-vector}) of $P$ as $h^*(x)=\sum_{k=0}^d a_k x^k$.
(There is another common way to define the $h^*$-polynomial, namely as the numerator of the Ehrhart series. That definition is equivalent to the one given above.)

In this paper we will compute the $h^*$-polynomials of polytopes by dissecting them into unimodular simplices, and then counting facets visible from an interior point of general position. This %method 
is a slight modification of the method of \cite{KV08_Barvinok_alg}, where $h$-vectors of triangulations are computed using the distribution of visible facets. Another inspiration is \cite{arithm_symedgepoly}, where the method of \cite{KV08_Barvinok_alg} is applied to unimodular triangulations to compute $h^*$-vectors. Our approach is stronger in that we do not require a unimodular triangulation, only a unimodular dissection.

Here a $d$-simplex $\conv\{\mathbf p_0, \mathbf p_1,\dots, \mathbf p_d\}\subset\mathbb R^n$ is called
\emph{unimodular} if its vertices satisfy $\mathbf p_0,\mathbf p_1, \dots, \mathbf p_d \in \mathbb{Z}^n$ and, in addition to $\mathbf p_1 - \mathbf p_0, \dots, \mathbf p_d-\mathbf p_0$ being linearly independent over $\mathbb R$, we also have that 
\begin{equation}
    \label{eq:unimodular}
\mathbb Z\langle\mathbf p_1 - \mathbf p_0, \dots, \mathbf p_d-\mathbf p_0\rangle=\mathbb R\langle\mathbf p_1 - \mathbf p_0, \dots, \mathbf p_d-\mathbf p_0\rangle\cap\mathbb Z^n.
\end{equation}
Whether this condition holds is independent of the choice of $\mathbf p_0$, and certainly of the rest of the order of the vertices.

Let $P$ be a $d$-dimensional lattice polytope and $\Delta_1, \dots, \Delta_s$ a dissection of $P$ into unimodular simplices. That is, $\Delta_1, \dots, \Delta_s$ are 
(relative)
interior-disjoint unimodular $d$-simplices such that $P=\Delta_1 \cup \dots \cup \Delta_s$.

Let $\mathbf{q}\in P$ be a point of general position with respect to the dissection 
$\Delta_1, \dots, \Delta_s$. By this, we mean that $\mathbf{q}$ is not contained in any facet-defining hyperplane of any of the simplices $\Delta_1, \dots, \Delta_s$.

For two points $\mathbf{p},\mathbf{q}\in \mathbb{R}^n$, let us denote by $[\mathbf p,\mathbf q]$ the closed segment connecting them, and let us denote by $(\mathbf{p},\mathbf{q})$ the relative interior of this segment.

We say that a point $\mathbf{p}\neq \mathbf q$ of a simplex $\Delta_i$ is \emph{visible} from $\mathbf q$ if $(\mathbf{p},\mathbf{q})$ is disjoint from $\Delta_i$. We say that a facet of $\Delta_i$ is visible from $\mathbf{q}$ if all points of the facet are visible from $\mathbf{q}$. 
It is easy to see that a facet of $\Delta_i$ is visible from $\mathbf q$ if and only if its hyperplane separates $\mathbf q$ from the interior of $\Delta_i$.
%\'atellenes cs\'ucs

Note that a point $\mathbf{p}\in \Delta_i$ is visible from $\mathbf{q}$ if and only if it belongs to a facet of $\Delta_i$ visible from $\mathbf{q}$. The ``if'' direction is clear by definition. For the ``only if'' direction, take a point 
$\mathbf{p}\in \Delta_i$ visible from $\mathbf{q}$. As  $(\mathbf{p},\mathbf{q})$ is disjoint from $\Delta_i$, we can take a supporting half-space of $\Delta_i$ that does not contain $\mathbf{q}$, moreover, $\mathbf{p}$ is on the supporting hyperplane. But then, by the general position of $\mathbf q$, there is also a facet-defining hyperplane containing $\mathbf{p}$ and separating $\mathbf{q}$ from the interior of $\Delta_i$. Now this facet will be visible from $\mathbf{q}$, and it contains $\mathbf{p}$.
% Ez egy kicsit elnagyolt, de annyi baj legyen.

For a simplex $\Delta$, let $\Vis_{\mathbf q}(\Delta)$ be the set of facets of $\Delta$ that are visible from $\mathbf{q}$.
Let also $H_\mathbf{q}(\Delta) = \Delta - \bigcup_{F\in \Vis_{\mathbf{q}}(\Delta)} F$. I.e., we remove the visible facets from $\Delta$. By the above, this is the same as removing all the visible points from $\Delta$.

\begin{lemma}\label{lem:visibility_disjoint}
Let $\Delta_1, \dots, \Delta_s$ be a dissection of the polytope $P$, and let $\mathbf{q}\in P$ be a point in general position with respect to the dissection. Then $P=H_\mathbf{q}(\Delta_1) \sqcup \dots \sqcup H_\mathbf{q}(\Delta_s)$ is a disjoint union.
\end{lemma}

We will refer to this decomposition as a \emph{half-open cover} of $P$.

\begin{remark}
The statement of Lemma \ref{lem:visibility_disjoint} is motivated by \cite[Theorem 3]{KV08_Barvinok_alg}, which deals with how an identity of indicator functions of polyhedra behaves under the operation $H_{\mathbf q}$, and by its application in \cite{arithm_symedgepoly}.
\end{remark}

\begin{proof}
The containment $P\supset H_\mathbf{q}(\Delta_1) \cup \dots \cup H_\mathbf{q}(\Delta_s)$ is clear from the definitions.
    Let now $\mathbf{p}\in P$ be an arbitrary point. We need to show that it is contained by exactly one set $H_\mathbf{q}(\Delta_i)$ of the half-open cover. If $\mathbf{p}=\mathbf{q}$, then by general position, $\mathbf{p}$ is %an 
    interior %point of 
    to one of the 
    %simplices 
    $\Delta_i$, whence it is contained only by $H_\mathbf{q}(\Delta_i)$.

    By renumbering the simplices of the dissection, we can suppose that $\mathbf{p}\in \Delta_1 \cap \dots \cap \Delta_r$ and $\mathbf p\not\in\Delta_{r+1}\cup\cdots\cup\Delta_s$, where $r\geq 1$. 
    Since the latter set is closed, $\mathbf p$ has a neighborhood $U\subset P$ (with respect to the relative topology of $P$)
    that is disjoint from it.
    We claim that if $\mathbf{p}\neq \mathbf{q}$, then the open segment $(\mathbf{p},\mathbf{q})$ intersects exactly one of  $\Delta_1,\ldots,\Delta_r$.
    %these simplices. 
    We start by showing that it intersects at least one of those simplices. By convexity, we have $(\mathbf p,\mathbf q)\subset P=\Delta_1\cup \dots \cup \Delta_s$; we also have that a sub-segment of $(\mathbf p,\mathbf q)$ (`near' $\mathbf p$) lies in $U$. Therefore
    $(\mathbf{p},\mathbf{q})$ needs to intersect a simplex incident to $\mathbf{p}$.
    
    Now suppose that $(\mathbf{p},\mathbf{q})$ intersects more than one of
    %simplex among 
    $\Delta_1, \dots, \Delta_r$.
    If the segment intersects $\Delta_i$ at $\mathbf{x}\neq \mathbf{p}$ and $\Delta_j$ at $\mathbf{y}\neq \mathbf{p}$, then by the convexity of simplices, we have $[\mathbf{x},\mathbf{p}] \subset \Delta_i$ and $[\mathbf{y},\mathbf{p}] \subset \Delta_j$. Hence one of the points $\mathbf{x}$ and $\mathbf{y}$, say $\mathbf x$, 
    is in $\Delta_i \cap \Delta_j$. 
    Now $\mathbf{p}\in \Delta_i \cap \Delta_j$ 
    implies $[\mathbf x,\mathbf p]\subset\Delta_i\cap\Delta_j$, in particular we see that $[\mathbf x,\mathbf p]$ lies along the boundary of $\Delta_i$. That means that
    the whole segment $[\mathbf{p},\mathbf{q}]$ is in %the 
    a facet-defining
    hyperplane of $\Delta_i$,
    %\cap \Delta_j$, 
    contradicting the general position of $\mathbf{q}$.
    
    Hence there is exactly one simplex $\Delta_i$, with $i\le r$, such that $(\mathbf{p},\mathbf{q})$ intersects $\Delta_i$. 
    %This means that 
    I.e., $\Delta_i$ is the only simplex containing $\mathbf{p}$ such that $\mathbf{p}$ is not visible from $\mathbf{q}$. Thus, in the half-open cover, $H_\mathbf{q}(\Delta_i)$ is the unique set that contains $\mathbf p$.
\end{proof}

The following well-known lemma plays a crucial role for us. We include a short proof for completeness.

\begin{lemma}\label{lemma:interior_point_integer_coordinates}
    Let $\Delta = \conv\{\mathbf{p}_0, \dots, \mathbf{p}_d\}$ be a unimodular $d$-simplex. For any positive integer $t$, a point $\mathbf{p}\in t\cdot\Delta$ is a lattice point if and only if $\mathbf{p} = \sum_{i=0}^d \mu_i \mathbf{p}_i$, where
    $\sum_{i=0}^d\mu_i=t$ and
    each $\mu_i$ is a non-negative integer.
\end{lemma}

\begin{proof}
	If $\mathbf{p}\in t\cdot \Delta$ then $\mathbf{p}=\sum_{i=0}^d \mu_i \mathbf{p}_i$, where $\mu_0, \dots, \mu_d \geq 0$ and  $\sum_{i=1}^d \mu_i = t$. This can be re-written as $\sum_{i=0}^d \mu_i (\mathbf{p}_i-\mathbf{p}_0) = \mathbf{p} - t\cdot \mathbf{p}_0$. The right hand side here is an integer vector 
	if and only if $\mathbf p$ is a lattice point.
By \eqref{eq:unimodular}, this is also equivalent to the claim that the left hand side is not only in the $\mathbb R$-span, but also in the $\mathbb Z$-span of $\{\mathbf{p}_i-\mathbf{p}_0\mid1\le i\le d\}$; by linear independence, this happens exactly when
	the coefficients $\mu_i$ are integer.
\end{proof}

Putting the ingredients together leads to the following proposition, which generalizes \cite[Proposition 2.1]{arithm_symedgepoly} to unimodular dissections.

\begin{prop} \label{prop:h^*-vector_from_visible_facets}
Let $\Delta_1, \dots, \Delta_s$ be a dissection of the $d$-dimensional lattice polytope $P\subset\mathbb R^n$ into unimodular simplices, and let $\mathbf{q}\in P$ be a point in general position with respect to the dissection. 
Then the $h^*$-polynomial $h^*(x)=h^*_d x^d +  \dots + h^*_1 x + h^*_0$ of $P$ 
has the coefficients
\[h^*_i = |\{ j \mid 1\le j\le s\text{ and } |\Vis_\mathbf{q}(\Delta_j)| = i\}|.\]
\end{prop}

\begin{remark}
The unimodularity of the dissection is crucial for the above statement to hold.
\end{remark}

\begin{proof}
By Lemma \ref{lem:visibility_disjoint}, we have the disjoint union (half-open cover) \[P=H_\mathbf{q}(\Delta_1) \sqcup \dots \sqcup H_\mathbf{q}(\Delta_s).\]
%, which is a disjoint union.
Let $t\in\mathbb{N}$ be a positive integer, and let us compute the cardinality $\varepsilon_P(t)=|(t\cdot P) \cap \mathbb{Z}^n|$. We have $t\cdot P = t\cdot H_\mathbf{q}(\Delta_1) \sqcup \dots \sqcup t\cdot H_\mathbf{q}(\Delta_t)$, which is still a disjoint union, hence it is enough to determine the number of lattice points in $t\cdot H_\mathbf{q}(\Delta)$ for a unimodular simplex $\Delta$. We claim that if $|\Vis_\mathbf{q}(\Delta)|= k$, then
$|(t\cdot H_\mathbf{q}(\Delta)) \cap \mathbb{Z}^n|=C_k(t)=\binom{t+d-k}{d}$.

First, suppose that $|\Vis_\mathbf{q}(\Delta)|=0$. Then we need to count lattice points within $t\cdot \Delta$. By Lemma \ref{lemma:interior_point_integer_coordinates}, the number of lattice points equals the number of non-negative integer $(d+1)$-tuples summing to $t$. This is indeed equal to $\binom{d+t}{d}=C_0(t)$ by a classical argument.

Now if $\Vis_\mathbf{q}(\Delta)=k$, then we remove $k$ facets from $\Delta$ to get $H_\mathbf{q}(\Delta)$. A (lattice) point lies along a certain facet of $\Delta$ if and only if in the convex combination of the vertices that produces it, the coefficient of the opposite vertex is zero. Hence in this case we need to count non-negative $(d+1)$-tuples where elements of a prescribed set of $k$ entries are strictly positive. The number of such tuples is $\binom{d+t-k}{d}=C_k(t)$ by another classical argument.

Now if we let $a_k= |\{\Delta_j \mid |\Vis_\mathbf{q}(\Delta_j)| = k\}|$ and summarize our count of lattice points across the pieces of the half-open cover, we get $\varepsilon_P(t)=\sum_{k=0}^d  a_k C_k(t)$ for each non-negative integer $t$. Hence indeed $h^*_k = a_k$ for all $k$.
\end{proof}

\section{Proofs of Theorems \ref{thm:sym_edge_h-vector}, \ref{thm:sym_edge_h-vector_fixed_order}, and \ref{thm:interior_poly_fixed_edge_order}}\label{sec:proofs}

Let us use the notation $\mathbf{x}_e$ for the vector $\mathbf{1}_h-\mathbf{1}_t\in\mathbb R^V$, where $e=\overrightarrow{th}$ is a directed edge in some graph on the vertex set $V$.

\begin{proof}[Proof of Theorem \ref{thm:sym_edge_h-vector}]
    Our proof is a very slight modification of the proof of \cite[Proposition 4.6]{arithm_symedgepoly}, that concerns a special case: a concrete family of triangulations of symmetric edge  polytopes of complete bipartite graphs.
	The main engine of the proof is Proposition \ref{prop:h^*-vector_from_visible_facets}.
	
	Let us fix a dissecting tree set $\mathcal{T}$ and a vertex $v\in V= V(\mathcal{G})$.
	We look for
	a point $\mathbf{q}$ in the interior of $P_{\mathcal{G}}$, such that for each tree $T\in\mathcal T$
	the number of facets, visible from $\mathbf q$, of the simplex $\tilde{\mathcal{Q}}_T$ of the dissection is equal to the number of edges of $T$ pointing away from $v$.
	
	We choose $\mathbf q=\sum_{u\in V}q_u\mathbf1_u$ so that $q_v > 0$ and for $u\neq v$ we have $q_u < 0$; moreover, we let $\sum_{u\in V} q_u = 0$.
	As $P_\mathcal{G}$ contains a ball around the origin within the subspace $\{\mathbf z\in \mathbb{R}^V\mid\sum_{u\in V} z_u=0\}$, by multiplying $\mathbf q$ with an appropriately small constant, we may also suppose that $\mathbf q \in P_\mathcal{G}$.
	We show that this $\mathbf q$
	has the properties that we are looking for.
	
	Let $T\in \mathcal{T}$ be a (directed) spanning tree from the dissecting tree set. The simplex $\tilde{\mathcal{Q}}_T$ has two types of facets. One facet, $\mathcal{Q}_{T}=\Conv\{ \mathbf{x}_f \mid f\in T\}$, corresponds to removing the origin. The rest of the facets correspond to removing one of the tree edges, i.e., they take the form $\tilde{\mathcal{Q}}_{T-e}=\Conv(\{\mathbf{0}\}\cup\{ \mathbf{x}_f \mid f\in T-e\})$ for some $e\in T$. 
	
	A facet ${\mathcal{Q}}_{T}$ is never visible from $\mathbf q$
	since it
	lies along the boundary of $P_\mathcal{G}$. 
	On the other hand, a  facet $\tilde{\mathcal{Q}}_{T-e}$ is visible from $\mathbf q$ if any only if $e$ is pointing away from $v$ in $T$. 
	
	To show the latter claim, let us consider the hyperplane of $\tilde{\mathcal{Q}}_{T-e}$ and the fundamental cut $C^*(T,e)$. Let $V_0$ and $V_1$ be the two shores of the cut, with $V_1$ containing the head of $e$. Take the linear functional $h\colon\mathbb{R}^V \to \mathbb{R}$ with $h(\mathbf{1}_u)=1$ if $u\in V_1$ and $h(\mathbf{1}_u)=0$ if $u\in V_0$. (This uniquely determines the values of $h$ on $\mathbb{R}^V$.)
	Clearly $h(\mathbf{0})=0$, moreover, $h(\mathbf{x}_f)=0$ for each $f\in T-e$ since both endpoints of these edges are on the same shore of $C^*(T,e)$. Hence the hyperplane $\{\mathbf x\in \mathbb{R}^V\mid h(\mathbf x)=0\}$ %indeed 
	contains $\tilde{\mathcal{Q}}_{T-e}$. Now as we also have $h(\mathbf{x}_e)=1$, the facet $\tilde{\mathcal{Q}}_{T-e}$ is visible from $\mathbf q$ if any only if $h(\mathbf q) < 0$.
	%We also have $h(\mathbf q)=\sum_{u\in V_1} q_u$. 
	As $\sum_{u\in V} q_u = 0$ with $q_v$ being the only positive coordinate, $h(\mathbf q)=\sum_{u\in V_1} q_u$ is negative if and only if $v\notin V_1$ (otherwise $h(\mathbf q)>0$). This is equivalent to $e$ pointing away from $v$ in $T$.
	
	As a byproduct we also see that $\mathbf q$ does not lie on any of the supporting hyperplanes of $\tilde{\mathcal Q}_T$ for any $T\in\mathcal T$, i.e., it is indeed in general position. This completes our proof.
\end{proof}

\begin{proof}[Proof of Theorem \ref{thm:sym_edge_h-vector_fixed_order}]
	We again apply Proposition \ref{prop:h^*-vector_from_visible_facets}.
	
	Let us fix a dissecting tree set $\mathcal{T}$ for $\cG$. We think of $\cG$ as a bidirected graph; that is, we take two oppositely directed edges for each undirected edge. We fix an ordering $<$ of these directed edges, % of $\cG$. 
	in other words we denote them by $e_1, \dots, e_m$ so that $e_1 < e_2 < \dots < e_m$ (here $m$ is even).
	Our strategy is to find a point $\mathbf{q}$ in the interior of $P_{\cG}$, such that for each simplex $\tilde{\mathcal{Q}}_T$ of the dissection, the number of facets of $\tilde{\mathcal{Q}}_T$ visible from $\mathbf q$ is equal to the internal semi-passivity of the tree $T$ with respect to $<$.
	
	Let $\mathbf q\in \mathbb{R}^V$ be the point
	$$\mathbf q= \sum_{i=1}^m \left(\frac{t^{i}}{\sum_{j=1}^m t^{j}}\right) \mathbf{x}_{e_i},$$
	where $t$ is sufficiently large. (For us, $t=2$ suffices.)
	Then $\mathbf q$ is 
	a convex combination of points of the form $\mathbf{x}_e$, with $e\in E$, hence by definition $\mathbf q\in P_\cG$.
	
	Next, we show that for any spanning tree simplex $\tilde{\mathcal{Q}}_T$, the number of facets visible from $\mathbf q$ is equal to the internal semi-passivity of the tree. At the same time, it will also turn out that $\mathbf q$ does not lie on any of the supporting hyperplanes of the simplices, i.e., $\mathbf q$ is in general position.
	
	Let $T\in \mathcal{T}$ be a tree of the dissecting tree set. The simplex $\tilde{\mathcal{Q}}_T$ has $\mathcal{Q}_T$ as a facet; just as before, this is never visible from $\mathbf q$. The rest of the facets are of the form $\tilde{\mathcal{Q}}_{T-e_k}=\conv\{ \mathbf{x}_{e_j} \mid e_j\in T-e_k\}$ for some edge $e_k\in T$.
		
	We show that the facet $\tilde{\mathcal{Q}}_{T-e_k}$ is visible from $\mathbf q$ if any only if $e_k$ is internally semi-passive in $T$. 
	Just like in the previous proof, 
	the hyperplane of $\tilde{\mathcal{Q}}_{T-e_k}$
	is described as the kernel of 
	the linear functional $h\colon\mathbb{R}^V \to \mathbb{R}$ that has $h(\mathbf{1}_u)=1$ if $u\in V_1$ and $h(\mathbf{1}_u)=0$ if $u\in V_0$; here $V_0,V_1\subset V$ are the shores of the fundamental cut $C^*(T,e_k)$, labeled so that $e_k$ points from $V_0$ to $V_1$. 
	We again have $h(\mathbf{x}_{e_k})=1$ and thus $\tilde{\mathcal{Q}}_{T-e_k}$ is visible from $\mathbf q$ if any only if $h(\mathbf q) < 0$.
	By linearity, we have 
	$$h(\mathbf q)=
	\left.\sum_{i=1}^mt^ih(\mathbf{x}_{e_i})\middle/\sum_{j=1}^m t^{j}\right..$$
	Here non-zero contributions to the numerator come from the edges of $C^*(T,e_k)$. Namely, those $e_i$
	that stand parallel to $e_k$ (that is, have their heads in $V_1$) have $h(\mathbf{x}_{e_i}) = 1$ and the edges of $C^*(T,e_k)$ standing opposite to $e_k$ 
	have $h(\mathbf{x}_{e_i}) = -1$. Hence $h(\mathbf q)<0$ 
	if and only if the largest edge of $C^*(T,e_k)$, according to $<$, stands opposite to $e_k$, i.e., if and only if $e_k$ is internally semi-passive in $T$.
\end{proof}

\begin{proof}[Proof of Theorem \ref{thm:interior_poly_fixed_edge_order}]
	We can apply the same construction for an interior point  $\mathbf{q}\in\mathcal Q_G$ as in the proof of Theorem \ref{thm:sym_edge_h-vector_fixed_order}, only now using the vectors $\mathbf{x}_e$ for edges of $G$ (and an arbitrary order). 
	%This will give us a $\mathbf{q}\in\mathcal{Q}_G$. 
	The rest of the proof carries over word-by-word.
\end{proof}

\section{Dissections, embedding activities, and some open problems}
%Open problems}
\label{s:open_problems}

\subsection{A closer look at dissecting tree sets}

We hope we have convinced the reader
that dissecting tree sets are quite remarkable objects. In this section we take a closer look at them, and 
indicate directions for future research.

First, it would be interesting to see a graph-theoretic characterization of dissecting tree sets.

\begin{problem}
\label{jokerdes}
Give a combinatorial characterization of dissecting tree sets for undirected graphs and semi-balanced digraphs.
\end{problem}

We saw that the distributions of many activity statistics agree for all dissecting tree sets. We wonder if the reverse of this statement is true in some sense.

\begin{question}
If for a set $\mathcal{T}$ of spanning trees
(either in an undirected graph or in a semi-balanced digraph), 
the distribution of
the internal semi-activity statistic agrees with the $h^*$-vector 
(of the relevant polytope)
for any fixed edge ordering, does it follow that $\mathcal{T}$ is a dissecting tree set? 
\end{question}

Even though we are unaware of an answer to Problem \ref{jokerdes},
we can at least give a simple 
sufficient condition for a set of spanning trees 
that ensures interior disjointness of the corresponding simplices.
%to form a dissecting tree set. I.e., the following holds.
We state it in the case of semi-balanced digraphs, which means that it can also be applied facet-by-facet to the symmetric edge polytope of any graph.

\begin{prop}\label{prop:dissecting_necessary_condition}
Let $\mathcal{T}$ be a set of spanning trees of the semi-balanced digraph $G$ such that for any $T_1, T_2\in\mathcal{T}$ there exists a cut $C^*$ in $G$ such that $T_1\cap C^*\subseteq (C^*)^+$ and $T_2\cap C^*\subseteq (C^*)^-$. (Here by $(C^*)^+$ we mean the edges of $C^*=(V_0,V_1)$ pointing from the shore $V_0$ to $V_1$, and by $(C^*)^-$ we mean the edges of $C^*$ pointing from $V_1$ to $V_0$.)
Then the simplices $\{\mathcal{Q}_T \mid T\in\mathcal{T}\}$ are interior disjoint.
\end{prop}

\begin{proof}
If $\mathcal{T}$ is a set of spanning trees so that the above `cut condition' holds, then for any two trees $T_1,T_2\in\mathcal{T}$, it suffices to find a hyperplane separating the interiors of $\mathcal{Q}_{T_1}$ and $\mathcal{Q}_{T_2}$. 

Let $C^*$ be a cut,
with shores $V_0$ and $V_1$, as stipulated in the Proposition.
We define the linear functional $h\colon\mathbb R^V\to\mathbb R$ by $h(\mathbf{1}_u)=1$ for $u\in V_1$ and $h(\mathbf{1}_u)=0$ for $u\in V_0$. Then $h(\mathbf x_e)\ne 0$ if and only if $e\in C^*$, furthermore $h(\mathbf x_e)=1$ when $e\in(C^*)^+$ and $h(\mathbf x_e)=-1$ when $e\in(C^*)^-$. Hence for each $p\in \mathcal{Q}_{T_1}$ we have $h(p)\geq 0$, and for each $p\in \mathcal{Q}_{T_2}$ we have $h(p)\leq 0$; moreover, for interior points, we have strict inequalities. Therefore indeed, $\mathcal{Q}_{T_1}$ and $\mathcal{Q}_{T_2}$ are interior disjoint.
\end{proof}

Note the level of similarity between this condition and Postnikov's necessary and sufficient local condition for triangulations of root polytopes \cite[Lemma 12.6]{alex} (which is extended to the semi-balanced case in \cite[Lemma 8.7]{semibalanced}).

\subsection{Non-robustness of embedding activities}
\label{ss:embedding_passivities}

Important motivation 
of the present paper comes from our earlier work \cite{semibalanced,sym_ribbon}, where we gave formulas that are analogous to Theorems \ref{thm:interior_poly_fixed_edge_order} and \ref{thm:sym_edge_h-vector_fixed_order}, but use semi-activities defined via a ribbon structure (embedding semi-activities) instead of semi-activities defined via a fixed edge order. 
In this section, we show that the formulas using embedding semi-activities are less robust than the formulas using semi-activities with respect to a fixed order, in that they do not work for \emph{any} dissecting tree set. 
(Unlike the formulas of Theorems \ref{thm:interior_poly_fixed_edge_order} and \ref{thm:sym_edge_h-vector_fixed_order}.) 
The background of this phenomenon
%non-robustness 
is unclear, and we pose it as an open problem to explain it.

Let us start with the necessary definitions. For more examples and explanations, see \cite[Section 5]{semibalanced}.

Let $G$ be a directed graph. A \emph{ribbon structure} of $G$ is an assignment of a cyclic ordering of the (in- and out-) edges incident to $v$ for each vertex $v$. For an edge $vu$, we denote by $vu^+$ the edge following $vu$ in the cyclic order around $v$. In addition to the ribbon structure, let us also fix a vertex $v_0$ and an edge $v_0v_1$ of $G$ (we do not care about the orientation of $v_0v_1$). We call the pair $(v_0,v_0v_1)$ the \emph{basis}. 

Let $T$ be a spanning tree of $G$. Given a ribbon structure and the basis
$(v_0,v_0v_1)$, 
the \emph{tour of $T$}, originally defined by Bernardi \cite{Bernardi_first}, is a sequence of node-edge pairs with the following rule. The first node-edge pair is $(v_0,v_0v_1)$. For a node-edge pair $(u,uv)$, if $uv\notin T$, then the next node-edge pair is $(u,uv^+)$. If $uv\in T$, then the next node-edge pair is $(v,vu^+)$. The sequence ends when the next node-edge pair would again be $(v_0,v_0v_1)$. (Note that the orientations do not matter in the definition.) By \cite[Lemma 5]{Bernardi_Tutte}, each incident node-edge pair occurs exactly once in this sequence.

If our digraph $G$ is semi-balanced, then we can use our ribbon structure and the resulting
tours of our trees to define a specific dissecting tree set. We call a spanning tree $T$ of a 
%semi-balanced digraph $G$ 
directed graph 
a \emph{Jaeger tree}, if in the tour of $T$, for every non-edge $\overrightarrow{uv}$ of $T$, $(u,\overrightarrow{uv})$ precedes $(v,\overrightarrow{uv})$. 
%in the tour of $T$. 
It is proved in \cite{semibalanced} that (for any ribbon structure and basis of a semi-balanced digraph) Jaeger trees form a dissecting tree set.
Example \ref{ex:coeulerian} is an instance of this result.

Let us now return to arbitrary digraphs and define embedding semi-activities.
For any spanning tree $T$, the tour of $T$ induces the following ordering $\leq_T$ on the edges of $G$: For 
$\overrightarrow{t_1h_1}$ and $\overrightarrow{t_2h_2}$, we put $\overrightarrow{t_1h_1}\leq_T \overrightarrow{t_2h_2}$ if $(t_1,t_1h_1)$ precedes $(t_2,t_2h_2)$ in the tour of $T$.

We say that an edge $e\in T$ is \emph{internally embedding semi-passive} in $T$ (with respect to the ribbon structure and basis $(v_0,v_0v_1)$) if it is internally semi-passive with respect to the edge ordering $\leq_T$ in the sense of Definition \ref{def:semiactive}. We define the \emph{internal embedding semi-passivity} of $T$ as the number of its internally embedding semi-passive edges. As in general different trees $T$ induce different orders $\leq_T$, internal embedding semi-passivities are not determined by a fixed edge ordering. Internal embedding semi-passivities are the analogues, and in some sense generalizations, of the embedding passivities of Bernardi \cite{Bernardi_Tutte} that he used to give an alternative description of the Tutte polynomial.

We also define another notion of activity: Let us fix a vertex $v_0$ of the directed graph $G$, and call it the base point. For a spanning tree $T$, we say that an edge $e\in T$ is \emph{internally basepoint-passive} if $e$ points away from $v_0$ in $T$, furthermore the fundamental cut $C^*(T,e)$ is not a directed cut (in $G$).

\begin{remark}
\label{rem:bidirected_activity}
Note that in a bidirected graph no cut is directed, whence for bidirected graphs, basepoint-passivity of $e$ in $T$ is equivalent to $e$ pointing away from $v_0$ in $T$. That is exactly the notion of passivity used in Theorem \ref{thm:sym_edge_h-vector}.
\end{remark}

It is proved in \cite[Lemma 6.4]{semibalanced} that for edges of Jaeger trees of semi-balanced digraphs, embedding semi-passivity is equivalent to basepoint-passivity (where we use the same ribbon structure and basis for defining the Jaeger trees and the embedding semi-passivities, and we use the vertex from the basis
%same $v_0$ 
to define basepoint-passivities).
Moreover, in \cite{semibalanced} we proved that the generating function of internal embedding semi-passivities for Jaeger trees gives the $h^*$-polynomial of the root polytope $\mathcal{Q}_G$.
However, it is easy to see that in general the two activity notions are different.

Here we show that if we consider an arbitrary dissecting tree set for our semi-balanced digraph, then neither internal embedding semi-passivities, nor basepoint-passivities give the $h^*$-polynomial of $\mathcal{Q}_G$ anymore.

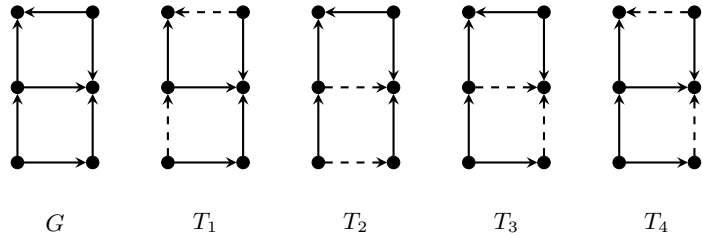
\begin{figure}
	\begin{center}
	\begin{tikzpicture}[scale=.2]
	\begin{scope}[shift={(-20,0)}]
	%\node [] (0) at (-2,10) {$s$};
	\node [circle,fill,scale=.5,draw] (1) at (0,10) {};
	\node [circle,fill,scale=.5,draw] (2) at (5,10) {};
	\node [circle,fill,scale=.5,draw] (3) at (0,5) {};
	\node [circle,fill,scale=.5,draw] (4) at (5,5) {};
	\node [circle,fill,scale=.5,draw] (5) at (0,0) {};		
	\node [circle,fill,scale=.5,draw] (6) at (5,0) {};		
	%\node [] (7) at (20,7) {$v$};	
	\path [thick,<-,>=stealth] (1) edge [left] node {} (2);
	\path [thick,->,>=stealth] (3) edge [above] node {} (4);
	\path [thick,->,>=stealth] (5) edge [right] node {} (6);
	\path [thick,->,>=stealth] (3) edge [below] node {} (1);
	\path [thick,<-,>=stealth] (4) edge [right] node {} (2);
	\path [thick,->,>=stealth] (5) edge [below] node {} (3);
	\path [thick,->,>=stealth] (6) edge [above] node {} (4);
	\node [] (0) at (2.5,-4) {\small $G$};
	\end{scope}
	\begin{scope}[shift={(-10,0)}]
	%\node [] (0) at (-2,10) {$s$};
	\node [circle,fill,scale=.5,draw] (1) at (0,10) {};
	\node [circle,fill,scale=.5,draw] (2) at (5,10) {};
	\node [circle,fill,scale=.5,draw] (3) at (0,5) {};
	\node [circle,fill,scale=.5,draw] (4) at (5,5) {};
	\node [circle,fill,scale=.5,draw] (5) at (0,0) {};		
	\node [circle,fill,scale=.5,draw] (6) at (5,0) {};		
	%\node [] (7) at (20,7) {$v$};	
	\path [thick,dashed,<-,>=stealth] (1) edge [left] node {} (2);
	\path [thick,->,>=stealth] (3) edge [above] node {} (4);
	\path [thick,->,>=stealth] (5) edge [right] node {} (6);
	\path [thick,->,>=stealth] (3) edge [below] node {} (1);
	\path [thick,<-,>=stealth] (4) edge [right] node {} (2);
	\path [thick,dashed,->,>=stealth] (5) edge [below] node {} (3);
	\path [thick,->,>=stealth] (6) edge [above] node {} (4);
	\node [] (0) at (2.5,-4) {\small $T_1$};
	\end{scope}
	\begin{scope}[shift={(0,0)}]
	%\node [] (0) at (-2,10) {$s$};
	\node [circle,fill,scale=.5,draw] (1) at (0,10) {};
	\node [circle,fill,scale=.5,draw] (2) at (5,10) {};
	\node [circle,fill,scale=.5,draw] (3) at (0,5) {};
	\node [circle,fill,scale=.5,draw] (4) at (5,5) {};
	\node [circle,fill,scale=.5,draw] (5) at (0,0) {};		
	\node [circle,fill,scale=.5,draw] (6) at (5,0) {};		
	%\node [] (7) at (20,7) {$v$};	
	\path [thick,<-,>=stealth] (1) edge [left] node {} (2);
	\path [thick,dashed,->,>=stealth] (3) edge [above] node {} (4);
	\path [thick,dashed,->,>=stealth] (5) edge [right] node {} (6);
	\path [thick,->,>=stealth] (3) edge [below] node {} (1);
	\path [thick,<-,>=stealth] (4) edge [right] node {} (2);
	\path [thick,->,>=stealth] (5) edge [below] node {} (3);
	\path [thick,->,>=stealth] (6) edge [above] node {} (4);
	\node [] (0) at (2.5,-4) {\small $T_2$};
	\end{scope}
	\begin{scope}[shift={(10,0)}]
	%\node [] (0) at (-2,10) {$s$};
	\node [circle,fill,scale=.5,draw] (1) at (0,10) {};
	\node [circle,fill,scale=.5,draw] (2) at (5,10) {};
	\node [circle,fill,scale=.5,draw] (3) at (0,5) {};
	\node [circle,fill,scale=.5,draw] (4) at (5,5) {};
	\node [circle,fill,scale=.5,draw] (5) at (0,0) {};		
	\node [circle,fill,scale=.5,draw] (6) at (5,0) {};		
	%\node [] (7) at (20,7) {$v$};	
	\path [thick,<-,>=stealth] (1) edge [left] node {} (2);
	\path [thick,dashed,->,>=stealth] (3) edge [above] node {} (4);
	\path [thick,->,>=stealth] (5) edge [right] node {} (6);
	\path [thick,->,>=stealth] (3) edge [below] node {} (1);
	\path [thick,<-,>=stealth] (4) edge [right] node {} (2);
	\path [thick,->,>=stealth] (5) edge [below] node {} (3);
	\path [thick,dashed,->,>=stealth] (6) edge [above] node {} (4);
	\node [] (0) at (2.5,-4) {\small $T_3$};
	\end{scope}
	\begin{scope}[shift={(20,0)}]
	%\node [] (0) at (-2,10) {$s$};
	\node [circle,fill,scale=.5,draw] (1) at (0,10) {};
	\node [circle,fill,scale=.5,draw] (2) at (5,10) {};
	\node [circle,fill,scale=.5,draw] (3) at (0,5) {};
	\node [circle,fill,scale=.5,draw] (4) at (5,5) {};
	\node [circle,fill,scale=.5,draw] (5) at (0,0) {};		
	\node [circle,fill,scale=.5,draw] (6) at (5,0) {};		
	%\node [] (7) at (20,7) {$v$};	
	\path [thick,dashed,<-,>=stealth] (1) edge [left] node {} (2);
	\path [thick,->,>=stealth] (3) edge [above] node {} (4);
	\path [thick,->,>=stealth] (5) edge [right] node {} (6);
	\path [thick,->,>=stealth] (3) edge [below] node {} (1);
	\path [thick,<-,>=stealth] (4) edge [right] node {} (2);
	\path [thick,->,>=stealth] (5) edge [below] node {} (3);
	\path [thick,->,>=stealth,dashed] (6) edge [above] node {} (4);
	\node [] (0) at (2.5,-4) {\small $T_4$};
	\end{scope}
	\end{tikzpicture}
	\end{center}
	\caption{A semi-balanced digraph (first panel), and four of its spanning trees (panels 2-5) that form a triangulating tree set.}
	\label{fig:counterex}
\end{figure}

Take the semi-balanced digraph $G$ depicted in the left panel of Figure \ref{fig:counterex}. 
We claim that the trees $T_1, \dots, T_4$ of Figure \ref{fig:counterex} form a dissecting tree set for $G$.
We have already seen in Example \ref{ex:semibalanced_fixed} that a dissecting tree set for this graph has to have four trees.
Hence by Proposition \ref{prop:dissecting_necessary_condition}, it suffices to find a cut for any two of the trees such that one of them only intesects the cut in edges going in one of the directions, while the other tree only intersects the cut in edges going in the other direction.
It is not hard to find such cuts for the trees $T_1, \dots, T_4$. For example, for $T_1$ and $T_2$, the elementary cut consisting of the three horizontal edges is suitable.

We note that the set of trees $T_1, \dots, T_4$ is even triangulating. For this, one can check that the condition obtained in \cite[Lemma 8.7]{semibalanced} holds.

As we saw in Example \ref{ex:semibalanced_fixed}, the $h^*$-polynomial of $\mathcal{Q}_G$ is $x^2+2x+1$.
However, notice that if we choose the lower left vertex of $G$ as the basepoint $v_0$, then both $T_1$ and $T_2$ have $0$ edges pointing away from $v_0$ such that their fundamental cut is not directed. 
That is, there are two trees with vanishing basepoint-passivity, whereas the constant term being $1$ in the $h^*$-polynomial would predict just one such tree.
Hence the statistic of basepoint-passivities does not give the $h^*$-polynomial for an arbitrary dissection, indeed not even for an arbitrary triangulation.

Now let us endow $G$ with the ribbon structure induced by the positive orientation of the plane. Moreover, 
let us specify a basis $(v_0,v_0v_1)$ by
keeping
the lower left vertex as $v_0$ and letting $v_1$ be the lower right vertex. Then 
one can check that $T_2$ has internal embedding semi-passivity $0$, while the trees $T_1$, $T_3$, and $T_4$ have internal embedding semi-passivity $1$. Hence we do not get the correct $h^*$-polynomial for this notion of activity, either.

Regarding symmetric edge polytopes $P_\cG$ of undirected graphs $\cG$, we saw in Theorem \ref{thm:sym_edge_h-vector} (cf.\ Remark \ref{rem:bidirected_activity}) that basepoint-passivities do give the $h^*$-polynomial for any dissecting tree set.
Let us now show that internal embedding semi-passivities do not necessarily yield the $h^*$-polynomial of $P_\cG$.

For this, let us return to the graph $K_3$. Figure \ref{fig:triangle_counterex} shows its unique dissecting tree set (cf.\ Example \ref{ex:K_3_computation}),
together with the orderings $\leq_T$ associated to the spanning trees $T$ when 
considering the ribbon structure induced by the positive orientation of the plane, the lower right vertex as base vertex, and base edge pointing from the upper vertex to the lower right vertex.
Notice that the first and third spanning trees both have internal embedding semi-passivity $2$; hence, as $x^2$ has coefficient $1$ in $h^*_{K_3}$ (by Example \ref{ex:K_3_computation}), this statistic does not give the correct $h^*$-polynomial. As there is only one possible choice of dissecting tree set for $K_3$, in fact, it does not give the $h^*$-polynomial for \emph{any} dissecting tree set.

\begin{figure}
    \centering
    \begin{tikzpicture}[scale=1]
    \tikzstyle{o}=[circle,fill,scale=.8,draw]
	
	\begin{scope}[shift={(-3.2,0)}]
	\node [o] (1) at (0,0) {};
	\node [o] (2) at (1,1.5) {};
	\node [o] (3) at (2,0) {};
	
	\node [] (7) at (2.3,-0.1) {\small $v_0$};	
	\path [thick,->,>=stealth,dashed] (1) edge [left,bend left=20] node {\small 4} (2);
	\path [thick,->,>=stealth,dashed] (2) edge [left,bend left=20] node {\small 1} (1);
	\path [ultra thick,->,>=stealth,] (1) edge [above,bend left=20] node {\small 5} (3);
	\path [thick,->,>=stealth,dashed] (3) edge [below,bend left=20] node {\small 6} (1);
	\path [ultra thick,->,>=stealth,] (2) edge [right,bend left=20] node {\small 2} (3);
	\path [thick,->,>=stealth,dashed] (3) edge [right,bend left=20] node {\small 3} (2);
	\end{scope}
	\begin{scope}[shift={(0,0)}]
	\node [o] (1) at (0,0) {};
	\node [o] (2) at (1,1.5) {};
	\node [o] (3) at (2,0) {};
	\node [] (7) at (2.3,-0.1) {\small $v_0$};	
	\path [thick,->,>=stealth,dashed] (1) edge [left,bend left=20] node {\small 6} (2);
	\path [thick,->,>=stealth,dashed] (2) edge [left,bend left=20] node {\small 3} (1);
	\path [thick,->,>=stealth,dashed] (1) edge [below,bend left=20] node {\small 5} (3);
	\path [ultra thick,->,>=stealth,] (3) edge [below,bend left=20] node {\small 4} (1);
	\path [thick,->,>=stealth,dashed] (2) edge [left,bend left=20] node {\small 2} (3);
	\path [ultra thick,->,>=stealth,] (3) edge [left,bend left=20] node {\small 1} (2);
	\end{scope}
	\begin{scope}[shift={(3.2,0)}]
	\node [o] (1) at (0,0) {};
	\node [o] (2) at (1,1.5) {};
	\node [o] (3) at (2,0) {};
	\node [] (7) at (2.3,-0.1) {\small $v_0$};	
	\path [thick,->,>=stealth,dashed] (1) edge [left,bend left=20] node {\small 6} (2);
	\path [ultra thick,->,>=stealth,] (2) edge [left,bend left=20] node {\small 5} (1);
	\path [thick,->,>=stealth,dashed] (1) edge [below,bend left=20] node {\small 3} (3);
	\path [ultra thick,->,>=stealth,] (3) edge [below,bend left=20] node {\small 2} (1);
	\path [thick,->,>=stealth,dashed] (2) edge [left,bend left=20] node {\small 4} (3);
	\path [thick,->,>=stealth,dashed] (3) edge [left,bend left=20] node {\small 1} (2);
	\end{scope}
	\begin{scope}[shift={(-3.2,-2.4)}]
	\node [o] (1) at (0,0) {};
	\node [o] (2) at (1,1.5) {};
	\node [o] (3) at (2,0) {};
	\node [] (7) at (2.3,-0.1) {\small $v_0$};	
	\path [ultra thick,->,>=stealth,] (1) edge [left,bend left=20] node {\small 2} (2);
	\path [thick,->,>=stealth,dashed] (2) edge [left,bend left=20] node {\small 3} (1);
	\path [ultra thick,->,>=stealth,] (1) edge [above,bend left=20] node {\small 5} (3);
	\path [thick,->,>=stealth,dashed] (3) edge [below,bend left=20] node {\small 6} (1);
	\path [thick,->,>=stealth,dashed] (2) edge [right,bend left=20] node {\small 4} (3);
	\path [thick,->,>=stealth,dashed] (3) edge [right,bend left=20] node {\small 1} (2);
	\end{scope}
	\begin{scope}[shift={(0,-2.4)}]
	\node [o] (1) at (0,0) {};
	\node [o] (2) at (1,1.5) {};
	\node [o] (3) at (2,0) {};
	\node [] (7) at (2.3,-0.1) {\small $v_0$};	
	\path [ultra thick,->,>=stealth,] (1) edge [left,bend left=20] node {\small 4} (2);
	\path [thick,->,>=stealth,dashed] (2) edge [left,bend left=20] node {\small 5} (1);
	\path [thick,->,>=stealth,dashed] (1) edge [below,bend left=20] node {\small 3} (3);
	\path [thick,->,>=stealth,dashed] (3) edge [below,bend left=20] node {\small 6} (1);
	\path [thick,->,>=stealth,dashed] (2) edge [left,bend left=20] node {\small 2} (3);
	\path [ultra thick,->,>=stealth,] (3) edge [left,bend left=20] node {\small 1} (2);
	\end{scope}
	\begin{scope}[shift={(3.2,-2.4)}]
	\node [o] (1) at (0,0) {};
	\node [o] (2) at (1,1.5) {};
	\node [o] (3) at (2,0) {};
	\node [] (7) at (2.3,-0.1) {\small $v_0$};	
	\path [thick,->,>=stealth,dashed] (1) edge [left,bend left=20] node {\small 2} (2);
	\path [ultra thick,->,>=stealth,] (2) edge [left,bend left=20] node {\small 1} (1);
	\path [thick,->,>=stealth,dashed] (1) edge [below,bend left=20] node {\small 3} (3);
	\path [thick,->,>=stealth,dashed] (3) edge [below,bend left=20] node {\small 6} (1);
	\path [ultra thick,->,>=stealth,] (2) edge [left,bend left=20] node {\small 4} (3);
	\path [thick,->,>=stealth,dashed] (3) edge [left,bend left=20] node {\small 5} (2);
	\end{scope}
	\end{tikzpicture}
    \caption{Edge orderings
    associated to various directed spanning trees %$T$ 
    of $K_3$ via a ribbon structure.}
    \label{fig:triangle_counterex}
\end{figure}
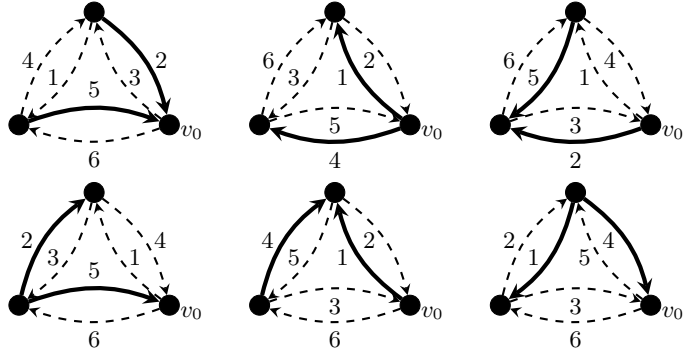

\begin{problem}
Characterize spanning tree statistics that yield the $h^*$-poly\-nomial of the root polytope of a semi-balanced digraph/symmetric edge polytope for any dissecting tree set. 
\end{problem}

\begin{problem}
What is the reason that for the root polytope of a semi-balanced digraph, internal embedding semi-passivities give the $h^*$-polynomial in the case of the Jaeger tree dissection, but not for an arbitrary dissection?
\end{problem}

\bibliographystyle{plain}
\bibliography{Bernardi}

\end{document}